\documentclass[a4paper,11pt]{article}
\usepackage[utf8]{inputenc}
\usepackage{a4wide}
\usepackage{algorithm}
\usepackage{algorithmic}
\usepackage{latexsym,amsfonts,amsmath,amssymb,mathrsfs,url,amsthm}
\usepackage{aliascnt}
\usepackage{mathtools}
\usepackage{dsfont}
\usepackage{color,graphicx}
\usepackage{lipsum}
\usepackage{subcaption}
\usepackage{hyperref}
\usepackage{cleveref}
\usepackage{xcolor}
\usepackage{placeins}

\newtheorem{theorem}{Theorem}[section]

\newaliascnt{lemma}{theorem}
\newtheorem{lemma}[lemma]{Lemma}
\aliascntresetthe{lemma}

\newaliascnt{example}{theorem}

\aliascntresetthe{example}

\newaliascnt{remark}{theorem}
\newtheorem{remark}[remark]{Remark}
\aliascntresetthe{remark}

\newaliascnt{definition}{theorem}
\newtheorem{definition}[definition]{Definition}
\aliascntresetthe{definition}

\newaliascnt{corollary}{theorem}
\newtheorem{corollary}[corollary]{Corollary}
\aliascntresetthe{corollary}

\newaliascnt{proposition}{theorem}
\newtheorem{proposition}[proposition]{Proposition}
\aliascntresetthe{proposition}

\crefname{theorem}{theorem}{theorems}
\crefname{lemma}{lemma}{lemmas}
\crefname{example}{example}{examples}
\crefname{remark}{remark}{remarks}
\crefname{definition}{definition}{definitions}
\crefname{corollary}{corollary}{corollaries}
\crefname{proposition}{proposition}{propositions}

\mathtoolsset{showonlyrefs}

\newcommand{\bsc}{\boldsymbol{c}}
\newcommand{\bsh}{\boldsymbol{h}}

\newcommand{\bsk}{\boldsymbol{k}}

\newcommand{\bsx}{\boldsymbol{x}}
\newcommand{\bsy}{\boldsymbol{y}}
\newcommand{\bsz}{\boldsymbol{z}}
\newcommand{\bsalpha}{\boldsymbol{\alpha}}
\newcommand{\bsb}{\boldsymbol{b}}

\newcommand{\PP}{\mathbb{P}}
\newcommand{\QQ}{\mathbb{Q}}
\newcommand{\RR}{\mathbb{R}}
\newcommand{\FF}{\mathbb{F}}
\newcommand{\TT}{\mathbb{T}}
\newcommand{\ZZ}{\mathbb{Z}}

\newcommand{\Ocal}{\mathcal{O}}
\newcommand{\Pcal}{\mathcal{P}}
\newcommand{\Scal}{\mathcal{S}}

\allowdisplaybreaks

\title{Algebraic constructions of point sequences with quasi-uniform two-dimensional projections\thanks{The work of T.G.\ was supported by JSPS KAKENHI Grant Number JP26K00620.}}
\author{Takashi Goda\thanks{Graduate School of Engineering, The University of Tokyo, 7-3-1 Hongo, Bunkyo-ku, Tokyo 113-8656, Japan (\url{goda@frcer.t.u-tokyo.ac.jp})}}
\date{\today}

\begin{document}

\maketitle

\begin{abstract}
    Motivated by sequential space-filling designs for computer experiments, we study algebraic constructions of extensible point sets in the $d$-dimensional unit cube whose two-dimensional coordinate projections are all quasi-uniform. Our two constructions share a common parametrization in terms of finite configurations of distinct rational directions on the projective line $\PP^1(\QQ)$. First, using a cubic number field, we construct explicit Kronecker sequences for which the mesh ratios of all two-dimensional coordinate projections remain uniformly bounded over every initial segment of length $N\ge 2$. Second, using a real quadratic field, a split prime, and a compatible $p$-adic embedding, we construct nested rank-1 lattice designs with the same uniform projection property at every nesting level. The proofs combine algebraic norm estimates with transference principles between simultaneous and dual Diophantine approximation, yielding lower bounds for the separation radii and upper bounds for the covering radii, both of optimal order in the number of points, uniformly over all coordinate pairs. We also investigate how the choice of rational projective coefficients affects the resulting mesh ratios. This leads to a minimax problem for finite configurations on $\PP^1(\QQ)$, in which one seeks to minimize the maximum mesh ratio over all two-dimensional coordinate projections. These constructions provide extensible point sets with uniformly controlled bivariate geometry.
\end{abstract}



\section{Introduction}\label{sec:introduction}

Quasi-uniformity balances the covering and separation properties of a finite point set. For a finite indexed family $P\subset[0,1)^s$, let $h(P)$, $q(P)$, and $\rho(P)=h(P)/q(P)$ denote its covering radius, separation radius, and mesh ratio, respectively; precise definitions are introduced in \Cref{sec:preliminaries}. A family of $N$-point sets is \emph{quasi-uniform} when its mesh ratios remain bounded, equivalently when $h(P)\asymp q(P)\asymp N^{-1/s}$. Such sets are central in scattered data approximation, kernel methods, and space-filling experimental design; see, for example, \cite{PM12,PZ23,SW06,W05}.

In this paper, we study a projection-wise version of this property. Let $\Scal=(\bsx_n)_{n\geq0}\subset[0,1)^d$, write $P_N=(\bsx_n)_{n=0}^{N-1}$, and let $\rho_{ij}(P_N)$ be the mesh ratio of the indexed projection onto coordinates $i$ and $j$, with $\rho_{ij}(P_N)=\infty$ when a collision of two points occurs. We aim to construct explicit sequences such that
\begin{align}\label{eq:intro-projection-qu}
 \sup_{N\geq2}\max_{1\leq i<j\leq d}\rho_{ij}(P_N)<\infty,
\end{align}
and nested rank-1 lattice families satisfying the analogous condition uniformly over all nesting levels. Throughout, the projected radii are measured in the ordinary Euclidean geometry of $[0,1)^2$; periodic lattice geometry is used only as an analytic and computational tool.

The property \eqref{eq:intro-projection-qu} is not implied by quasi-uniformity in the ambient $d$-dimensional cube. For a rank-1 lattice point set $P_{N,\bsz}=\{\{n\bsz/N\} \mid 0\leq n<N\}$, its $(i,j)$-projection contains only $N/\gcd(N,z_i,z_j)$ distinct points. Even without collisions, a short vector in the projected dual lattice may confine the points to a few parallel lines. Such behavior was observed in our preceding work on full-dimensional space-filling lattice designs \cite{SG26}. 

From the viewpoint of computer experiments, space-filling designs are useful when an expensive deterministic simulator must be explored with a limited number of runs; see, for example, \cite{FLS06,JMY90,MM95,SWN03}. Projection-wise space-filling is particularly relevant when the response has low effective dimension: only a few input variables, or low-order interactions among them, may have a substantial influence on the response. Since the active variables are typically not known in advance, a design that is well spread in the full $d$-dimensional cube but has poor low-dimensional projections can be undesirable. This observation has motivated maximum projection designs and lattice-based designs with favorable projection properties \cite{H21,JGB15}. Closely related concerns also arise in quasi-Monte Carlo (QMC) constructions, where the quality of two-dimensional projections can strongly depend on the generating parameters \cite{JK08}. The present paper seeks explicit extensible constructions ensuring quasi-uniformity on every bivariate projection.

We further remark that extensibility of point sets is desirable in practice, particularly when the computational budget is not fixed in advance, as additional runs can be added without discarding previously evaluated design points. Nested (or sequential) space-filling designs for computer experiments and extensible lattice rules for QMC have also been studied; see, for instance, \cite{RHVD10} and \cite{DPW08}, respectively.

The two constructions we propose in this paper are parametrized in terms of a finite configuration of distinct rational directions
\[
 V=(v_1,\ldots,v_d),\quad v_j=[a_j:b_j]\in\PP^1(\QQ)=\left(\QQ^2\setminus\{(0,0)\}\right)/\QQ^\times,
\]
where $(a_j,b_j)$ is a fixed primitive integer representative. Distinctness is equivalent to $a_i b_j-a_j b_i\neq0$ for $i\neq j$. The same configuration is inserted into a cubic-field construction through $\alpha_j=a_j\theta+b_j\theta^2$ and into a quadratic-field construction through $\beta_j=a_j+b_j\omega$.

Our contributions are as follows.
\begin{enumerate}
\item \emph{Algebraic Kronecker sequences.} Let $K=\QQ(\theta)\subset\RR$ be a cubic number field, with $\theta$ an algebraic integer, and set $\alpha_j=a_j\theta+b_j\theta^2$. For every $i<j$, the elements $1,\alpha_i,\alpha_j$ form a $\QQ$-basis of $K$. An algebraic norm estimate gives dual bad approximability, and Khintchine transference together with the characterization in \cite{DGLPS25} yields $q_{ij}(P_N)\asymp h_{ij}(P_N)\asymp N^{-1/2}$ uniformly over $N\geq2$ and $i<j$. Thus, every bivariate projection of the resulting Kronecker sequence is quasi-uniform.

\item \emph{Nested rank-1 lattice sequences.} Let $L=\QQ(\omega)$ be real quadratic, let $p$ split in $L$, and choose a prime ideal $\mathfrak p$ above $p$ and the corresponding embedding $\iota_{\mathfrak p}:\Ocal_L\to\ZZ_p$. For $\beta_j=a_j+b_j\omega$ that are units at $\mathfrak p$, define $z_{m,j}\equiv\iota_{\mathfrak p}(\beta_j)\pmod{p^m}$. The rank-1 lattice point sets with $N=p^m$ and $\bsz=(z_{m,1},\ldots,z_{m,d})$ form a nested family as $m$ varies. A dual congruence at level $m$ forces $p^m$ to divide $N_{L/\QQ}(h_1\beta_i+h_2\beta_j)$, which gives a lower bound of order $p^{m/2}$ for every projected dual minimum. Planar lattice geometry and a boundary-transfer estimate then imply $q_{ij}(P_m)\asymp h_{ij}(P_m)\asymp p^{-m/2}$ uniformly over $m$ and $i<j$, ensuring projection-wise quasi-uniformity.

\item \emph{Extremal projective configurations.} Distinct directions guarantee bounded mesh ratios but not small constants. We therefore formulate a finite minimax problem on bounded-height subsets of $\PP^1(\QQ)$. For the nested lattice construction, each periodic projected mesh ratio is obtained exactly by Gauss reduction of a planar dual lattice; for the Kronecker construction, algebraic and truncated Diophantine criteria are used to screen candidates before direct geometric evaluation. This separates an all-level algebraic certificate from the quantitative optimization of the observed mesh ratios.
\end{enumerate}

Both constructions are intrinsically bivariate. For $d>2$, the cubic coordinates lie in $\QQ\theta+\QQ\theta^2$, while the quadratic coefficients lie in the two-dimensional space $L$; hence higher-dimensional projections need not be uniformly distributed. The fixed algebraic degree is precisely what permits arbitrary ambient dimension while retaining non-degeneracy for every coordinate pair.

The rest of this paper is organized as follows. \Cref{sec:preliminaries} collects the geometric, lattice, projective, and Diophantine preliminaries. The cubic and quadratic constructions are developed in \Cref{sec:kronecker,sec:nested}. \Cref{sec:optimization} formulates the extremal coefficient problem, \Cref{sec:numerics} reports the resulting computations, and \Cref{sec:conclusion} concludes the paper.


\section{Preliminaries}\label{sec:preliminaries}

We write $\|\cdot\|_2$ and $\|\cdot\|_\infty$ for the Euclidean and maximum norms, $\|x\|=\min_{k\in\ZZ}|x-k|$, and denote by $\{\cdot\}$ fractional parts componentwise. The torus $\TT^s=\RR^s/\ZZ^s$ is identified with $[0,1)^s$ and endowed with $d_{\TT^s}(\bsx,\bsy)=\min_{\bsk\in\ZZ^s}\|\bsx-\bsy-\bsk\|_2$. Implicit constants may depend on fixed algebraic data and on $d$, but not on the number of points or the nesting level.

\subsection{Euclidean and periodic quasi-uniformity}

Let $\Pcal=(\bsx_n)_{n=0}^{N-1}\subset[0,1)^s$ be an indexed family, so that collisions are retained. Define
\begin{align*}
 h(\Pcal)\coloneqq \sup_{\bsx\in[0,1)^s}\min_{0\le n<N}\|\bsx-\bsx_n\|_2,\quad 
 q(\Pcal)\coloneqq \frac{1}{2}\min_{0\le m<n<N}\|\bsx_m-\bsx_n\|_2,\quad
 \rho(\Pcal)\coloneqq \frac{h(\Pcal)}{q(\Pcal)}.
\end{align*}
We set $q(\Pcal)=0$ and $\rho(\Pcal)=\infty$ when two labeled points coincide. The periodic quantities $h_{\rm per}$, $q_{\rm per}$, and $\rho_{\rm per}$ are defined by replacing $\|\bsx-\bsy\|_2$ with $d_{\TT^s}(\bsx,\bsy)$. A sequence of finite families is quasi-uniform if its Euclidean mesh ratios are bounded.

\begin{lemma}\label{lem:boundary-transfer}
For every finite indexed family $\Pcal\subset[0,1)^s$,
\begin{align}\label{eq:boundary-transfer-radii}
 q_{\rm per}(\Pcal)\le q(\Pcal),\quad
 h_{\rm per}(\Pcal)\le h(\Pcal)\le(1+\sqrt{s})h_{\rm per}(\Pcal).
\end{align}
Consequently, $\rho(\Pcal)\le(1+\sqrt{s})\rho_{\rm per}(\Pcal)$ whenever $q_{\rm per}(\Pcal)>0$.
\end{lemma}

\begin{proof}
The first two inequalities follow immediately from $d_{\TT^s}\le\|\cdot\|_2$. Thus, only the last displayed inequality requires proof. Put $r=h_{\rm per}(\Pcal)$. If $r<1/2$, fix $\varepsilon>0$ so that $r+\varepsilon <1/2$ and clip an arbitrary $\bsx\in[0,1)^s$ coordinatewise to $[r+\varepsilon,1-r-\varepsilon]$. The clipped point moves by at most $\sqrt{s}(r+\varepsilon)$, and any periodic covering point within distance $r$ has a representative in the original cube $[0,1)^s$. Hence its Euclidean distance from $\Pcal$ is at most $(1+\sqrt{s})r+\sqrt{s}\varepsilon$. Letting $\varepsilon \downarrow 0$ proves the desired estimate in this case. If $r\ge1/2$, choose a periodic covering point for $\bsc=(1/2,\ldots,1/2)$. Then, its periodic distance from every point of $[0,1)^s$ agrees with the Euclidean distance. By using the triangle inequality, we obtain
\[ \min_{0\le n<N}\|\bsx-\bsx_n\|_2\le \|\bsx-\bsc\|_2+r\le \frac{\sqrt{s}}{2}+r\le (1+\sqrt{s})r, \]
for any $\bsx\in [0,1)^s$. This completes the proof of the last displayed inequality.
\end{proof}

For a coordinate pair $i<j$, let $\Pcal^{(i,j)}$ be the indexed projection and write $h_{ij}(\Pcal)=h(\Pcal^{(i,j)})$, $q_{ij}(\Pcal)=q(\Pcal^{(i,j)})$, and $\rho_{ij}(\Pcal)=\rho(\Pcal^{(i,j)})$, with analogous periodic notation.

\subsection{Rank-1 lattices and planar reduction}\label{subsec:planar-reduction}

For a full-rank lattice $\Lambda\subset\RR^s$, let $\Lambda^\perp$, $\det(\Lambda)$, $\lambda_1(\Lambda)$, and $\mu(\Lambda)$ denote its dual, determinant, first minimum, and covering radius. If $\ZZ^s\subset\Lambda$ and $P(\Lambda)=\Lambda \cap [0,1)^s$, then $h_{\rm per}(P(\Lambda))=\mu(\Lambda)$ and $q(P(\Lambda))\ge\lambda_1(\Lambda)/2$; thus \Cref{lem:boundary-transfer} converts periodic lattice bounds into Euclidean bounds.

For $N\ge2$ and $\bsz\in\ZZ^s$, set
\[
 \Lambda_{N,\bsz}=\ZZ^s+\frac{\bsz}{N}\ZZ,
 \quad
 P_{N,\bsz}=\Lambda_{N,\bsz} \cap [0,1)^s=\left\{\left\{\frac{n\bsz}{N}\right\}\mid 0\le n<N\right\}.
\]
Here, $P_{N,\bsz}$ is called a \emph{rank-1 lattice point set with generating vector $\bsz$} consisting of $M=N/\gcd(N,z_1,\ldots,z_s)$ distinct points in $[0,1)^s$, and 
\begin{align}\label{eq:rank-one-determinant}
 \det(\Lambda_{N,\bsz})=M^{-1}.
\end{align}
Moreover, $\Lambda_{N,\bsz}^{\perp}=\{\bsh\in\ZZ^s:\bsh\cdot\bsz\equiv0\pmod N\}$. In dimension two, if $z_1$ is invertible modulo $N$, then $P_{N,(z_1,z_2)}=P_{N,(1,a)}$ after reindexing, where $a\equiv z_2z_1^{-1}\pmod N$. Its dual congruence lattice and the corresponding generator matrix are given by
\begin{align}\label{eq:dual-congruence-lattice}
 \Gamma_{N,a}=\left\{(h_1,h_2)\in\ZZ^2 \mid h_1+ah_2\equiv0\pmod N\right\},
 \quad
 B_{N,a}=\begin{pmatrix}N&-a\\0&1\end{pmatrix},
\end{align}
respectively, with determinant $N$.

A basis $(\bsb_1,\bsb_2)$ of a planar lattice is \emph{Gauss reduced} if
\begin{align}\label{eq:gauss-reduced}
 \|\bsb_1\|_2\le\|\bsb_2\|_2,
 \quad 2|\bsb_1\cdot\bsb_2|\le\|\bsb_1\|_2^2.
\end{align}
We choose the sign of $\bsb_2$ so that $0\le2\bsb_1\cdot\bsb_2\le\|\bsb_1\|_2^2$. Then $\bsb_1$ is shortest, the triangle $(0,\bsb_1,\bsb_2)$ is Delaunay, and its circumradius gives the covering radius.

\begin{proposition}\label{prop:exact-periodic-mesh-ratio}
Let $(\bsb_1,\bsb_2)$ be a Gauss-reduced basis of $\Gamma_{N,a}$ with the preceding sign convention. Then, we have
\begin{align}
 q_{\rm per}(P_{N,(1,a)})&=\frac{\|\bsb_1\|_2}{2N},\label{eq:exact-periodic-separation}\\
 h_{\rm per}(P_{N,(1,a)})&=\frac{\|\bsb_1\|_2\|\bsb_2\|_2\|\bsb_2-\bsb_1\|_2}{2N^2},\label{eq:exact-periodic-covering}\\
 \rho_{\rm per}(P_{N,(1,a)})&=\frac{\|\bsb_2\|_2\|\bsb_2-\bsb_1\|_2}{N}.\label{eq:exact-periodic-mesh}
\end{align}
\end{proposition}

\begin{proof}
In dimension two, a lattice is similar to its dual. Indeed, if $\Gamma=B\ZZ^2$ and $J$ denotes rotation through $\pi/2$, then $B^{-T} = (\det B)^{-1}JBJ^T$. Since $J^T\ZZ^2=\ZZ^2$, this gives, up to orientation, $\Gamma^{\perp}=\det(\Gamma)^{-1}J\Gamma$. As $\Gamma_{N,a}=\Lambda^{\perp}_{N,(1,a)}$ and $\det(\Gamma_{N,a})=N$, it follows that $\Lambda_{N,(1,a)}$ is isometric to $N^{-1}\Gamma_{N,a}$. Since $\det(\Lambda_{N,(1,a)})=1/N$ and $N\ge2$, the planar Hermite bound gives $\lambda_1(\Lambda_{N,(1,a)})<1$. Hence, a shortest nonzero vector is not an integer vector and represents a nonzero difference in the quotient $\Lambda_{N,(1,a)}/\mathbb Z^2$. The first formula follows from the shortest vector, while the covering radius of $\Gamma_{N,a}$ is the circumradius of the Delaunay triangle $(0,\bsb_1,\bsb_2)$, namely $\|\bsb_1\|_2\|\bsb_2\|_2\|\bsb_2-\bsb_1\|_2/(2N)$.
\end{proof}

\begin{corollary}\label{cor:dual-minimum-quasi-uniformity}
If $\lambda_1(\Gamma_{N,a})\ge c\sqrt N$ for some $c>0$, then
\[
 q(P_{N,(1,a)})\ge\frac{c}{2}N^{-1/2},\quad
 h(P_{N,(1,a)})\le C(c)N^{-1/2},
\]
and hence $\rho(P_{N,(1,a)})$ is bounded in terms of $c$ only.
\end{corollary}

\begin{proof}
The lower bound on $q(P_{N,(1,a)})$ follows immediately from 
\[ q(P_{N,(1,a)})\ge q_{\rm per}(P_{N,(1,a)})= \frac{\lambda_1(\Gamma_{N,a})}{2N}\ge \frac{c}{2}N^{-1/2}.\]
Let us consider $h(P_{N,(1,a)})$. For a Gauss-reduced basis, the planar Hermite bound gives $\|\bsb_1\|_2\le\eta\sqrt N$ with $\eta=(2/\sqrt3)^{1/2}$, while $\sin\angle(\bsb_1,\bsb_2)\ge\sqrt3/2$ and $\det\Gamma_{N,a}=N$ give $\|\bsb_2\|_2\le A\sqrt N$ with $A=2/(\sqrt3c)$. It follows from the triangle inequality that
\[ \|\bsb_2-\bsb_1\|_2 \le \|\bsb_2\|_2 +\|\bsb_1\|_2\le (A+\eta)\sqrt{N}. \]
Substituting these bounds into \eqref{eq:exact-periodic-covering} and then applying \Cref{lem:boundary-transfer} completes the proof. For example, one may take $C(c)=(1+\sqrt2)\eta A(\eta+A)/2$.
\end{proof}

\subsection{Projective coefficients and Diophantine approximation}

The rational projective line is $\PP^1(\QQ)=(\QQ^2\setminus\{0\})/\QQ^\times$. Every $v=[a:b]$ has a primitive integer representative, unique up to sign; we choose the one satisfying $a>0$, or $(a,b)=(0,1)$, and define the height by
\begin{align}\label{eq:projective-height}
 H(v)\coloneqq \max\{|a|,|b|\}.
\end{align}

\begin{definition}\label{def:projective-configuration}
A rational projective coefficient configuration of size $d$ is an ordered tuple $V=(v_1,\ldots,v_d)$ of distinct points $v_j=[a_j:b_j]\in\PP^1(\QQ)$, represented as above. Thus
\begin{align}\label{eq:projective-distinctness-determinant}
 a_i b_j-a_j b_i\neq0\quad(i\neq j).
\end{align}
Its height is $H(V)=\max_j H(v_j)$, and $\mathcal R_B=\{v\in\PP^1(\QQ):H(v)\le B\}$ will be our finite candidate set.
\end{definition}

For $\bsalpha\in\RR^s$, define
\begin{align}
 c_{\rm sim}(\bsalpha) \coloneqq \inf_{q\ge1}q^{1/s}\max_j\|q\alpha_j\|,\quad 
 c_{\rm dual}(\bsalpha) \coloneqq \inf_{\bsh\in\ZZ^s\setminus\{0\}}\|\bsh\|_\infty^s\|\bsh\cdot\bsalpha\|.
\end{align}
The following equivalence follows from the classical transference theorem for two homogeneous problems; see \cite[Chapter~V, Corollary to Theorem~II]{C57}. The proof there is quantitative: a positive lower bound for either constant yields a positive lower bound for the other depending only on that constant and on s.
\begin{theorem}\label{thm:khintchine-transference}
For every $\bsalpha\in\RR^s$, one has $c_{\rm sim}(\bsalpha)>0$ if and only if $c_{\rm dual}(\bsalpha)>0$.
\end{theorem}

For $\bsalpha\in\RR^s$, we denote the initial segment with length $N$ of the corresponding Kronecker sequence by
\begin{align}\label{eq:kronecker-initial-segment}
 K_N(\bsalpha)=(\{n\bsalpha\})_{n=0}^{N-1}.
\end{align}
We shall also use the following characterization from \cite{DGLPS25}.

\begin{theorem}\label{thm:kronecker-characterization}
The families $K_N(\bsalpha)$ are quasi-uniform in $[0,1)^s$ if and only if $c_{\rm sim}(\bsalpha)>0$. In that case, $q(K_N(\bsalpha))\gtrsim N^{-1/s}$ and $h(K_N(\bsalpha))\lesssim N^{-1/s}$ uniformly over $N\ge2$.
\end{theorem}


\section{Algebraic Kronecker sequences}
\label{sec:kronecker}

In this section, we construct, in every ambient dimension $d\geq2$, an explicit Kronecker sequence whose every two-dimensional coordinate projection is quasi-uniform.  The construction starts from a cubic number field and a rational projective coefficient configuration.  Its essential feature is that each pair of projective directions produces a basis of the same cubic field.  An algebraic norm estimate then gives a uniform dual Diophantine inequality, which is converted into simultaneous bad approximability by transference.

\subsection{A cubic norm estimate}
\label{subsec:cubic-norm-estimate}

Let $K$ be a cubic number field embedded in $\mathbb{R}$, and choose a real algebraic integer $\theta\in\Ocal_K$ such that $K=\QQ(\theta)$.
We denote the three embeddings of $K$ into $\mathbb{C}$ by
\[
 \sigma_0=\operatorname{id},\quad \sigma_1,\quad \sigma_2,
\]
where $\sigma_0$ is the fixed real embedding. The other two embeddings are
either both real or form a complex-conjugate pair. For $\gamma\in K$, we write
\[ N_{K/\QQ}(\gamma)=\prod_{\ell=0}^{2}\sigma_{\ell}(\gamma)\]
for its field norm. In particular, $N_{K/\QQ}(\gamma)\in\ZZ$ for every
$\gamma\in\Ocal_K$.

Let
\[
 V=(v_1,\ldots,v_d),
 \quad
 v_j=[a_j:b_j]\in\PP^1(\QQ),
\]
be a rational projective coefficient configuration in the sense of
\Cref{def:projective-configuration}, and let $(a_j,b_j)$ be its canonical
primitive integer representatives.  We associate with $V$ the algebraic
numbers
\begin{align}\label{eq:cubic-alpha-definition}
 \alpha_j\coloneqq a_j\theta+b_j\theta^2\in\Ocal_K,
 \quad j=1,\ldots,d,
\end{align}
and write $\boldsymbol{\alpha}(K,V) \coloneqq (\alpha_1,\ldots,\alpha_d).$

The projective distinctness condition has the following immediate algebraic
interpretation.

\begin{lemma}\label{lem:cubic-pair-basis}
For every $1\leq i<j\leq d$, the three elements $1,\alpha_i,\alpha_j$ form a $\QQ$-basis of $K$.
\end{lemma}

\begin{proof}
With respect to the basis $(1,\theta,\theta^2)$, the coefficient matrix of
$(1,\alpha_i,\alpha_j)$ is
\[
 \begin{pmatrix}
  1&0&0\\
  0&a_i&b_i\\
  0&a_j&b_j
 \end{pmatrix}.
\]
Its determinant is
$\Delta_{ij}=a_i b_j-a_j b_i$, which is nonzero by
\eqref{eq:projective-distinctness-determinant}.
\end{proof}

For $i<j$ and $\ell\in\{1,2\}$, set
\begin{align}\label{eq:cubic-archimedean-factor}
 B_{ij,\ell}
 \coloneqq 1+|\alpha_i|+|\alpha_j|
   +|\sigma_\ell(\alpha_i)|+|\sigma_\ell(\alpha_j)|,
\end{align}
and define
\begin{align}\label{eq:cubic-dual-certificate}
 \kappa_{ij}\coloneqq \frac{1}{B_{ij,1}B_{ij,2}}.
\end{align}
The next proposition gives a completely explicit certificate for dual bad
approximability. The constants may not be sharp, but their
dependence on the algebraic field and on the projective coefficients is
transparent.

\begin{proposition}\label{prop:cubic-dual-bound}
For every $1\leq i<j\leq d$ and every
$(h_1,h_2)\in\mathbb{Z}^2\setminus\{(0,0)\}$,
\begin{align}\label{eq:cubic-dual-inequality}
 \bigl\|h_1\alpha_i+h_2\alpha_j\bigr\|
 \geq
 \kappa_{ij}\,
 \max\{|h_1|,|h_2|\}^{-2}.
\end{align}
Consequently, $c_{\rm dual}(\alpha_i,\alpha_j)\geq\kappa_{ij}>0$.
\end{proposition}

\begin{proof}
Put $H\coloneqq \max\{|h_1|,|h_2|\}\geq1$, and choose $h_0\in\mathbb{Z}$ such that $\gamma\coloneqq h_0+h_1\alpha_i+h_2\alpha_j$ satisfies $|\gamma|=\|h_1\alpha_i+h_2\alpha_j\|$.
By \Cref{lem:cubic-pair-basis}, $\gamma\neq0$.  Moreover,
$\gamma\in\Ocal_K$, and hence
\[
 1\leq|N_{K/\QQ}(\gamma)|
 =|\gamma|\,|\sigma_1(\gamma)|\,|\sigma_2(\gamma)|.
\]
The choice of $h_0$ gives
\[
 |h_0|
 \leq |h_1\alpha_i+h_2\alpha_j|+\frac12
 \leq H\left(1+|\alpha_i|+|\alpha_j|\right).
\]
Therefore, for $\ell\in \{1,2\}$, we have
\[
 |\sigma_\ell(\gamma)|
 \leq |h_0|+|h_1|\,|\sigma_\ell(\alpha_i)|
                 +|h_2|\,|\sigma_\ell(\alpha_j)|
 \leq H B_{ij,\ell}.
\]
Substituting these bounds into the norm inequality yields
\[
 |\gamma|
 \geq \frac{1}{B_{ij,1}B_{ij,2}}H^{-2},
\]
which is nothing but \eqref{eq:cubic-dual-inequality}.
\end{proof}

It is convenient to aggregate the pairwise certificates into the
field-dependent functional
\begin{align}\label{eq:cubic-algebraic-cost}
 \mathcal{C}_K(V)
 \coloneqq \max_{1\leq i<j\leq d} B_{ij,1}B_{ij,2}
 =\left(\min_{1\leq i<j\leq d}\kappa_{ij}\right)^{-1}.
\end{align}
Thus
\begin{align}\label{eq:cubic-uniform-dual-bound}
 \bigl\|h_1\alpha_i+h_2\alpha_j\bigr\|
 \geq
 \mathcal{C}_K(V)^{-1}
 \max\{|h_1|,|h_2|\}^{-2}
\end{align}
holds simultaneously for every coordinate pair.  The quantity
$\mathcal{C}_K(V)$ will serve later as a certified, although generally
conservative, algebraic proxy for the worst projected mesh ratio.

\subsection{Quasi-uniform bivariate projections}
\label{subsec:cubic-main-theorem}

The $d$-dimensional Kronecker sequence generated by
$\boldsymbol{\alpha}(K,V)$ is
\begin{align}\label{eq:cubic-kronecker-sequence}
 \boldsymbol{x}_n
 \coloneqq \left(\{n\alpha_1\},\ldots,\{n\alpha_d\}\right),
 \quad n=0,1,2,\ldots,
\end{align}
and its first $N$ points, retained as an indexed family, are denoted by
\begin{align}\label{eq:cubic-kronecker-initial-family}
 \mathcal{K}_N(K,V)
 \coloneqq (\boldsymbol{x}_n)_{n=0}^{N-1}.
\end{align}
For each pair $i<j$, its $(i,j)$-projection is exactly the initial segment
$K_N(\alpha_i,\alpha_j)$ from
\eqref{eq:kronecker-initial-segment}.

\begin{theorem}\label{thm:kronecker-main}
Let $K=\QQ(\theta)\subset\mathbb{R}$ be a cubic number field generated by a
real algebraic integer $\theta$, and let
$V=(v_1,\ldots,v_d)$ be a rational projective coefficient configuration of
size $d\geq2$.  Then every two-dimensional coordinate projection of the
Kronecker sequence \eqref{eq:cubic-kronecker-sequence} is quasi-uniform in
the Euclidean unit square.  More precisely, there exist constants
$c_{K,V},C_{K,V}>0$ such that, for all $N\geq2$ and all
$1\leq i<j\leq d$,
\begin{align*}
 q_{ij}\left(\mathcal{K}_N(K,V)\right) \ge c_{K,V}N^{-1/2}\quad \text{and}\quad
 h_{ij}\left(\mathcal{K}_N(K,V)\right) \le C_{K,V}N^{-1/2}.
\end{align*}
In particular, we have
\begin{align*}
 \sup_{N\geq2}\max_{1\leq i<j\leq d}
 \rho_{ij}\left(\mathcal{K}_N(K,V)\right)<\infty.
\end{align*}
\end{theorem}

\begin{proof}
By \Cref{prop:cubic-dual-bound}, each vector $(\alpha_i,\alpha_j)$ is dually badly approximable. The transference shown in \Cref{thm:khintchine-transference} therefore implies that it is simultaneously badly approximable.  The characterization in \Cref{thm:kronecker-characterization}, applied with $s=2$, yields constants $c_{ij},C_{ij}>0$ for which
\[
 q\left(K_N(\alpha_i,\alpha_j)\right)
 \geq c_{ij}N^{-1/2},
 \quad
 h\left(K_N(\alpha_i,\alpha_j)\right)
 \leq C_{ij}N^{-1/2},
\]
for every $N\geq2$. Since there are only finitely many coordinate pairs, we may take
\[
 c_{K,V}\coloneqq \min_{i<j}c_{ij}>0,
 \quad
 C_{K,V}\coloneqq \max_{i<j}C_{ij}<\infty.
\]
This completes the proof.
\end{proof}

\begin{remark}\label{rem:cubic-effectivity}
The proof separates the qualitative and quantitative roles of the two Diophantine inequalities. The norm argument gives the explicit dual certificate
\[
 c_{\rm dual}(\alpha_i,\alpha_j)\geq\kappa_{ij}.
\]
Quantitative forms of the transference shown in \Cref{thm:khintchine-transference} and of the proof of \Cref{thm:kronecker-characterization} (from \cite{DGLPS25}) then make the constants $c_{K,V},C_{K,V}$ in Theorem~\ref{thm:kronecker-main} effective. Those transferred constants are typically too conservative to rank candidate configurations accurately.  In the extremal problem considered later, we use both the certified algebraic quantity $\mathcal{C}_K(V)$ and finite-segment geometric criteria.
\end{remark}

\begin{remark}
    Although every one-dimensional projection is uniformly distributed, its quasi-uniformity is a much subtler question. Each coordinate $\alpha_j=a_j\theta+b_j\theta^2$ is a cubic irrational, and the quasi-uniformity of its one-dimensional Kronecker sequence is equivalent to $\alpha_j$ having bounded continued-fraction partial quotients \cite{G24b}. It is widely conjectured that every real algebraic irrational of degree at least three has unbounded partial quotients; this question appears to go back to Khintchine \cite{ABD06}. Thus, conjecturally, none of the one-dimensional projections of our cubic-field construction is quasi-uniform, although every two-dimensional coordinate projection is quasi-uniform.
\end{remark}

\subsection{Explicit coefficient families}
\label{subsec:cubic-explicit-families}

\Cref{thm:kronecker-main} applies to any finite collection of distinct rational directions. A particularly simple family is obtained by choosing distinct integers $c_1,\ldots,c_d$ and setting
\begin{align}\label{eq:cubic-affine-projective-family}
 v_j=[1:c_j],
 \quad
 \alpha_j=\theta+c_j\theta^2.
\end{align}
Indeed, $\Delta_{ij}=c_j-c_i\neq 0$. This gives the following fully explicit construction.

\begin{corollary}
\label{cor:cubic-explicit-family}
Let $\theta$ be any real algebraic integer of degree three.  For every $d\geq2$, the Kronecker sequence generated by
\begin{align}\label{eq:cubic-simple-generator}
 \bsalpha_d =\left(\theta, \theta+\theta^2, \ldots, \theta+(d-1)\theta^2\right)
\end{align}
has quasi-uniform two-dimensional coordinate projections, uniformly over all initial segment lengths $N\geq2$.
\end{corollary}

Taking $\theta=\sqrt[3]{2}$ in \eqref{eq:cubic-simple-generator} produces a construction requiring no search and no auxiliary arithmetic choices.  Another natural configuration distributes four coefficient directions more symmetrically in $\PP^1(\mathbb{R})$. For instance, in $d=4$, we may take
\begin{align}\label{eq:cubic-cross-configuration}
 V_4^{\rm cross}
 \coloneqq \left([1:0],[0:1],[1:1],[1:-1]\right).
\end{align}
It gives $\bsalpha^{\rm cross} =\left(\theta,\theta^2,\theta+\theta^2,\theta-\theta^2\right)$. No optimality is asserted here; configurations such as \eqref{eq:cubic-cross-configuration} will be compared systematically in the later extremal analysis.

\begin{remark}
\label{rem:cubic-not-full-dimensional}
For $d>2$, all coordinates $\alpha_j$ belong to the two-dimensional $\QQ$-vector space $\QQ\theta+\QQ\theta^2$. Hence there is a nonzero $\boldsymbol{h}\in\mathbb{Z}^d$ such that $\boldsymbol{h}\cdot\boldsymbol{\alpha}=0$. The full $d$-dimensional Kronecker sequence is therefore contained in a proper subtorus and is not uniformly distributed in $[0,1)^d$. By contrast, \Cref{lem:cubic-pair-basis} shows that every pair $1,\alpha_i,\alpha_j$ is rationally independent.  The fixed degree of the number field is precisely what allows the ambient dimension $d$ to be arbitrary while retaining uniform control of every bivariate projection.
\end{remark}


\section{Nested rank-1 lattice sequences}
\label{sec:nested}

We next construct nested rank-1 lattice point sets whose bivariate coordinate projections are all quasi-uniform.  The compatibility across nesting levels is supplied by a $p$-adic embedding of a real quadratic field, while the lower bounds for the projected dual minima follow from divisibility of algebraic norms.  As in the previous cubic construction, the coordinates are parametrized in terms of a rational projective coefficient configuration.

\subsection{Split primes and admissible coefficient configurations}
\label{subsec:quadratic-local-data}

Let $L$ be a real quadratic field.  We fix an integral generator $\omega\in\Ocal_L$ such that
\begin{align}\label{eq:quadratic-ring-generator}
 L=\QQ(\omega),
 \quad
 \Ocal_L=\ZZ[\omega].
\end{align}
Let $\sigma_1,\sigma_2:L\to \RR$ denote its two real embeddings. For $\gamma\in L$, we write
\[
N_{L/\QQ}(\gamma)=\sigma_1(\gamma)\sigma_2(\gamma)
\]
for its field norm.
Let $p$ be a rational prime that splits in $L$. Thus, we have
\begin{align}\label{eq:quadratic-split-prime}
p\Ocal_L=\mathfrak p\,\overline{\mathfrak p},
\quad
\mathfrak p\neq\overline{\mathfrak p},
\end{align}
where $\mathfrak p$ and $\overline{\mathfrak p}$ are the two distinct prime ideals of $\Ocal_L$ lying above $p$. Because $p$ splits, the completion $L_{\mathfrak p}$ of $L$ with respect to the $\mathfrak{p}$-adic valuation is isomorphic to $\QQ_p$, the field of $p$-adic numbers. Fix such an isomorphism. Composing it with the canonical embedding of $\Ocal_L$ into the ring of integers of $L_{\mathfrak p}$ gives an embedding
\begin{align}\label{eq:quadratic-p-adic-embedding}
 \iota_{\mathfrak p}:\Ocal_L\to \mathbb{Z}_p.
\end{align}
For $m\geq1$, reduction modulo $p^m$ gives a compatible ring homomorphism
\begin{align}\label{eq:quadratic-residue-map}
 \iota_m:\Ocal_L\to\mathbb{Z}/p^m\mathbb{Z},
 \quad
 \iota_{m+1}(\beta)\equiv\iota_m(\beta)\pmod{p^m}.
\end{align}
Its kernel is $\mathfrak p^m$. For a non-zero ideal $\mathfrak{a}\subset \Ocal_L$, we denote its absolute norm by
\[ N(\mathfrak{a})\coloneqq |\Ocal_L/\mathfrak{a}|.\]
Recall that $N((\gamma))=|N_{L/\QQ}(\gamma)|$ for every non-zero $\gamma\in \Ocal_L$.

As in the previous cubic-field construction, let
\[
 V=(v_1,\ldots,v_d),
 \quad
 v_j=[a_j:b_j]\in\PP^1(\QQ),
\]
be a rational projective coefficient configuration, with canonical primitive
integer representatives $(a_j,b_j)$.  Define
\begin{align}\label{eq:quadratic-beta-definition}
 \beta_j\coloneqq a_j+b_j\omega\in\Ocal_L,
 \quad j=1,\ldots,d.
\end{align}
Since $\omega\notin\QQ$, projective distinctness is equivalent to pairwise
rational independence of the corresponding quadratic numbers:
\begin{align}\label{eq:quadratic-projective-independence}
 v_i\neq v_j
 \quad\iff\quad
 \frac{\beta_i}{\beta_j}\notin\QQ.
\end{align}
Indeed, $\beta_i=q\beta_j$ with $q\in\QQ$ is equivalent, after comparing the
coefficients of $1$ and $\omega$, to
$(a_i,b_i)=q(a_j,b_j)$.

\begin{definition}\label{def:p-admissible-configuration}
The configuration $V$ is called $\mathfrak p$-admissible if
\begin{align}\label{eq:p-admissibility}
 \beta_j\notin\mathfrak p
 \quad(j=1,\ldots,d),
\end{align}
or equivalently, if
$\iota_{\mathfrak p}(\beta_j)\in\mathbb{Z}_p^\times$ for every $j$.
\end{definition}

The admissibility condition ensures that every generating component is
invertible modulo every power of $p$, which is mild.  For a fixed
configuration, only finitely many rational primes divide
$\prod_j|N_{L/\QQ}(\beta_j)|$, whereas a quadratic field has infinitely many
split primes.  Hence every finite rational projective configuration is
$\mathfrak p$-admissible for infinitely many choices of a split prime and a
prime ideal above it.  In particular, if all $\beta_j$ are global units,
then admissibility holds for every split prime.

\subsection{The nested construction}
\label{subsec:nested-construction}

Assume from now on that $V$ is $\mathfrak p$-admissible.  For each $m\geq1$
and $j=1,\ldots,d$, let
\begin{align}\label{eq:nested-generator-residue}
 z_{m,j}\in\{0,1,\ldots,p^m-1\}
\end{align}
be the standard integer representative of $\iota_m(\beta_j)$, and put
\begin{align}\label{eq:nested-generator-vector}
 \boldsymbol{z}_m\coloneqq (z_{m,1},\ldots,z_{m,d}).
\end{align}
The associated rank-1 lattice point family is
\begin{align}\label{eq:nested-point-family}
 \mathcal{P}_m(L,\mathfrak p,V)
 \coloneqq \left\{
 \left\{\frac{n\boldsymbol{z}_m}{p^m}\right\}
 \mid 0\le n<p^m\right\}
 \subset[0,1)^d.
\end{align}

\begin{proposition}
\label{prop:nestedness}
For every $m\geq1$, the family $\mathcal{P}_m(L,\mathfrak p,V)$ consists of exactly $p^m$ distinct points, and each of its coordinate projections, of any positive dimension, also contains $p^m$ distinct labeled points.  Moreover, $\mathcal{P}_m(L,\mathfrak p,V) \subset \mathcal{P}_{m+1}(L,\mathfrak p,V)$ as unlabeled point sets.
\end{proposition}

\begin{remark}
    In fact, every one-dimensional coordinate projection is the regular grid
    \[ \{0,1/p^m,\ldots,(p^m-1)/p^m\}. \]
    Hence its Euclidean separation and covering radii are $1/(2p^m)$ and $1/p^m$, respectively, and its mesh ratio is identically $2$.
\end{remark}

\begin{proof}[Proof of \Cref{prop:nestedness}]
Admissibility gives $p\nmid z_{m,j}$ for every $j$.  Hence $\gcd(p^m,z_{m,j_1},\ldots,z_{m,j_s})=1$ for every nonempty set of coordinates $\{j_1,\ldots,j_s\}$.  The distinctness statement follows from \eqref{eq:rank-one-determinant} and the discussion preceding it.

Compatibility of the residue maps gives $z_{m+1,j}\equiv z_{m,j}\pmod{p^m}$.  Therefore, for any $0\leq n<p^m$, we have
\[
 \left\{\frac{pn z_{m+1,j}}{p^{m+1}}\right\}
 =\left\{\frac{n z_{m+1,j}}{p^m}\right\}
 =\left\{\frac{n z_{m,j}}{p^m}\right\}
\]
for every coordinate $j$.  Thus the point with index $n$ at level $m$
coincides with the point with index $pn$ at level $m+1$, proving $\mathcal{P}_m(L,\mathfrak p,V) \subset \mathcal{P}_{m+1}(L,\mathfrak p,V)$.
\end{proof}

The nested family can be converted into a single infinite sequence by ordering first the points of $\mathcal{P}_1$, then the points of $\mathcal{P}_2\setminus\mathcal{P}_1$, and so on.  For instance, one may use lexicographic order within every newly added shell. The first $p^m$ points of the resulting sequence are exactly $\mathcal{P}_m$. Alternatively, one may order the points of $\mathcal{P}_{m+1}\setminus\mathcal{P}_m$ by farthest-point sampling, which greedily favors well-separated intermediate point sets and may improve their covering properties for $p^m<N<p^{m+1}$. 

For a coordinate pair $i<j$, define the projected dual congruence lattice
\begin{align}\label{eq:nested-projected-dual}
 \Gamma_{m}^{(i,j)}
 \coloneqq \left\{(h_1,h_2)\in\mathbb{Z}^2:
 h_1z_{m,i}+h_2z_{m,j}\equiv0\pmod{p^m}\right\}.
\end{align}
Since $z_{m,i}$ is invertible modulo $p^m$, this is the normal-form lattice
$\Gamma_{p^m,a_{m,ij}}$ from \eqref{eq:dual-congruence-lattice}, where $ a_{m,ij}\equiv z_{m,j}z_{m,i}^{-1}\pmod{p^m}$. In particular, $ \det\left(\Gamma_m^{(i,j)}\right)=p^m$.

\subsection{Algebraic lower bounds for the projected dual minima}
\label{subsec:quadratic-dual-minima}

The key is that a dual congruence at level $m$ forces divisibility by
$\mathfrak p^m$ in the quadratic field.

\begin{lemma}\label{lem:quadratic-norm-divisibility}
Let $i<j$, $m\geq1$, and
$(h_1,h_2)\in\Gamma_m^{(i,j)}\setminus\{(0,0)\}$.  Then
\begin{align}\label{eq:quadratic-norm-divisibility}
 p^m\mid
 N_{L/\QQ}(h_1\beta_i+h_2\beta_j),
\end{align}
and the norm on the right-hand side is nonzero.
\end{lemma}

\begin{proof}
Set $\gamma=h_1\beta_i+h_2\beta_j$.  By
\eqref{eq:nested-projected-dual} and the definition of the residue maps, it holds that
\[
 \iota_m(\gamma)=0.
\]
Since $\ker(\iota_m)=\mathfrak p^m$, we have
$\gamma\in\mathfrak p^m$.  The prime ideal $\mathfrak p$ has absolute norm
$p$, so the principal ideal $(\gamma)$ has norm divisible by $p^m$, which
proves \eqref{eq:quadratic-norm-divisibility}.  Finally, if $\gamma=0$, then
$\beta_i/\beta_j\in\QQ$, contrary to
\eqref{eq:quadratic-projective-independence}.
\end{proof}

For each pair $i<j$, define the archimedean norm factor
\begin{align}\label{eq:quadratic-archimedean-factor}
 D_{ij}
 \coloneqq \prod_{\ell=1}^2
 \left(|\sigma_\ell(\beta_i)|+|\sigma_\ell(\beta_j)|\right),
\end{align}
and aggregate these constants as
\begin{align}\label{eq:quadratic-algebraic-cost}
 \mathcal{D}_L(V)
 \coloneqq \max_{1\leq i<j\leq d}D_{ij}.
\end{align}

\begin{proposition}\label{prop:quadratic-dual-minimum}
For every $m\geq1$ and every $1\leq i<j\leq d$,
\begin{align}\label{eq:quadratic-dual-minimum}
 \lambda_1\left(\Gamma_m^{(i,j)}\right)
 \geq D_{ij}^{-1/2}p^{m/2}
 \geq \mathcal{D}_L(V)^{-1/2}p^{m/2}.
\end{align}
\end{proposition}

\begin{proof}
Let $(h_1,h_2)\in\Gamma_m^{(i,j)}\setminus\{(0,0)\}$ and put
$H\coloneqq \max\{|h_1|,|h_2|\}$.  By
\Cref{lem:quadratic-norm-divisibility},
\[
 p^m
 \leq\left|N_{L/\QQ}(h_1\beta_i+h_2\beta_j)\right|.
\]
On the other hand,
\begin{align*}
 \left|N_{L/\QQ}(h_1\beta_i+h_2\beta_j)\right|
 &=\prod_{\ell=1}^2
 \left|h_1\sigma_\ell(\beta_i)
       +h_2\sigma_\ell(\beta_j)\right|\\
 &\leq D_{ij}H^2.
\end{align*}
It follows that $H\geq D_{ij}^{-1/2}p^{m/2}$.  Since
$\|(h_1,h_2)\|_2\geq H$, taking the minimum over nonzero vectors proves the
claim.
\end{proof}

The inequality \eqref{eq:quadratic-dual-minimum} is the quadratic analogue
of the dual Diophantine estimate
\eqref{eq:cubic-dual-inequality}.  Here it applies directly to a sequence of
finite congruence lattices and has the optimal scale
$\det(\Gamma_m^{(i,j)})^{1/2}=p^{m/2}$.

\subsection{Quasi-uniform bivariate projections}
\label{subsec:quadratic-main-theorem}

We now combine the dual-minimum bound with the planar lattice geometry of
\Cref{subsec:planar-reduction}.  Recall the dimension-two Hermite constant $\eta=(2/\sqrt3)^{1/2}$, used in the proof of \Cref{cor:dual-minimum-quasi-uniformity}. For a pair $i<j$, set
\begin{align}\label{eq:quadratic-Aij-definition}
 A_{ij}\coloneqq \frac{2\sqrt{D_{ij}}}{\sqrt{3}}.
\end{align}

\begin{theorem}
\label{thm:nested-main}
Let $L=\QQ(\omega)$ be a real quadratic field with
$\Ocal_L=\mathbb{Z}[\omega]$, let $p$ split as in
\eqref{eq:quadratic-split-prime}, and fix the $p$-adic embedding
\eqref{eq:quadratic-p-adic-embedding}.  Let
$V=(v_1,\ldots,v_d)$ be a $\mathfrak p$-admissible rational
projective coefficient configuration of size $d\geq2$.  Then the rank-1 lattice point
sets \eqref{eq:nested-point-family} are nested, and all of their
bivariate coordinate projections are quasi-uniform in the Euclidean unit
square, uniformly over the nesting level.  More precisely, for all
$m\geq1$ and $i<j$,
\begin{align}
 q_{ij}\left(\mathcal{P}_m(L,\mathfrak p,V)\right)
 &\geq \frac{1}{2\sqrt{D_{ij}}}\,p^{-m/2},
 \label{eq:nested-main-separation}\\
 h_{ij}\left(\mathcal{P}_m(L,\mathfrak p,V)\right)
 &\leq
 \frac{(1+\sqrt{2})\eta A_{ij}(\eta+A_{ij})}{2}\,p^{-m/2},
 \label{eq:nested-main-covering}\\
 \rho_{ij}\left(\mathcal{P}_m(L,\mathfrak p,V)\right)
 &\leq
 (1+\sqrt{2})A_{ij}(\eta+A_{ij}).
 \label{eq:nested-main-mesh}
\end{align}
In particular,
\begin{align}\label{eq:nested-uniform-mesh}
 \sup_{m\geq1}\max_{1\leq i<j\leq d}
 \rho_{ij}\left(\mathcal{P}_m(L,\mathfrak p,V)\right)<\infty.
\end{align}
\end{theorem}

\begin{proof}
Nestedness was proved in \Cref{prop:nestedness}.  For a fixed pair $i<j$,
its projected point set is, after reindexing, the normal-form rank-1 lattice
point set $P_{p^m,(1,a_{m,ij})}$.  By
\Cref{prop:quadratic-dual-minimum}, its dual congruence lattice satisfies
\[
 \lambda_1\left(\Gamma_m^{(i,j)}\right)
 \geq c_{ij}\sqrt{p^m},
 \quad
 c_{ij}\coloneqq D_{ij}^{-1/2}.
\]
Applying \Cref{cor:dual-minimum-quasi-uniformity} with $N=p^m$ and $c=c_{ij}$ gives the stated bounds on $q_{ij}$ and $h_{ij}$. Moreover, the estimates in the proof of that corollary and \Cref{prop:exact-periodic-mesh-ratio} give
\[
 \rho_{\rm per}(P_m^{(i,j)})
 \le A_{ij}(\eta+A_{ij}).
\]
Hence, by \Cref{lem:boundary-transfer},
\[
 \rho_{ij}(P_m(L,\mathfrak{p},V))
 \le (1+\sqrt2)A_{ij}(\eta+A_{ij}).
\]
Taking the maximum over the finitely many coordinate pairs proves \eqref{eq:nested-uniform-mesh}.
\end{proof}

\begin{remark}
\label{rem:nested-periodic-euclidean}
The norm argument controls the intrinsic shape of the projected periodic
lattice.  In fact, a Gauss-reduced basis of
$\Gamma_m^{(i,j)}$ gives the exact periodic mesh ratio through
\eqref{eq:exact-periodic-mesh}.  The factor $1+\sqrt{2}$ in
the upper bounds on $h_{ij}$ and $\rho_{ij}$ enters only
when the periodic covering estimate is transferred to the ordinary unit
square.  For quantitative optimization we shall therefore distinguish the
exact periodic objective from the boundary-sensitive Euclidean objective.
\end{remark}

\subsection{Hensel lifts and a fully explicit family}
\label{subsec:hensel-explicit-family}

The $p$-adic construction can be implemented using only modular arithmetic.
Let
\begin{align}\label{eq:quadratic-minimal-polynomial}
 f(X)=X^2-TX+U\in\mathbb{Z}[X]
\end{align}
be the minimal polynomial of $\omega$.  Since $p$ splits in $L$,
$f$ has two distinct roots modulo $p$.  Choose one of them, say
$r_1\in\{0,\ldots,p-1\}$.  Hensel's lemma gives a unique
$r\in\mathbb{Z}_p$ satisfying
\begin{align}\label{eq:hensel-p-adic-root}
 f(r)=0,
 \quad
 r\equiv r_1\pmod p,
\end{align}
and the corresponding embedding is characterized by
$\iota_{\mathfrak p}(\omega)=r$.  If $r_m$ denotes the residue of $r$
modulo $p^m$, then
\begin{align}\label{eq:hensel-root-compatibility}
 f(r_m)\equiv0\pmod{p^m},
 \quad
 r_{m+1}\equiv r_m\pmod{p^m}.
\end{align}
More explicitly,
\begin{align}\label{eq:hensel-recursion}
 r_{m+1}=r_m+t_mp^m,
 \quad
 t_m\equiv
 -\frac{f(r_m)}{p^m}\left(f'(r_m)\right)^{-1}
 \pmod p,
\end{align}
where $t_m\in\{0,\ldots,p-1\}$.  For the coefficient
$\beta_j=a_j+b_j\omega$, one then has
\begin{align}\label{eq:hensel-generator-components}
 z_{m,j}\equiv a_j+b_jr_m\pmod{p^m}.
\end{align}
Thus, the nested generating vectors can be computed recursively without
performing arithmetic in a number-field package.

A canonical example is provided by the golden-ratio field.  Let
\begin{align}\label{eq:golden-field-data}
 L=\QQ(\sqrt{5}),
 \quad
 \omega=\frac{1+\sqrt{5}}{2},
 \quad
 f(X)=X^2-X-1.
\end{align}
The prime $p=11$ splits, and $r_1=4$ is a simple root of $f$ modulo $11$.
Let $r_m$ be its compatible Hensel lifts.  Choose
\begin{align}\label{eq:golden-beta-family}
 \beta_j\coloneqq \omega^{j-1},
 \quad j=1,\ldots,d.
\end{align}
These are global units, so the configuration is admissible at every prime.
Moreover, the associated projective directions are pairwise distinct.  For
$j\geq2$,
\begin{align}\label{eq:golden-fibonacci-coefficients}
 \omega^{j-1}=F_{j-2}+F_{j-1}\omega,
\end{align}
where $F_0=0$, $F_1=1$, and
$F_{n+1}=F_n+F_{n-1}$.  Consecutive Fibonacci numbers are coprime, so these
are already primitive coefficient representatives.

\begin{corollary}
\label{cor:golden-nested-family}
For every $d\geq2$, let $r_m$ be determined by
\eqref{eq:golden-field-data} and
\eqref{eq:hensel-root-compatibility}, with $r_1=4$.  Then the generating
vectors
\begin{align}\label{eq:golden-generating-vector}
 \boldsymbol{z}_m
 =\left(1,r_m,r_m^2,\ldots,r_m^{d-1}\right)
 \pmod{11^m}
\end{align}
define nested rank-1 lattice point sets with quasi-uniform
two-dimensional coordinate projections at every level $m\geq1$.
\end{corollary}

\begin{proof}
The residue of $\beta_j=\omega^{j-1}$ under the chosen embedding is
$r_m^{j-1}$ modulo $11^m$, so
\eqref{eq:golden-generating-vector} is a special case of
\eqref{eq:nested-generator-vector}.  The ratio
$\beta_i/\beta_j=\omega^{i-j}$ is not rational for $i\neq j$, and every
$\beta_j$ is a global unit.  The result follows from
\Cref{thm:nested-main}.
\end{proof}

The family \eqref{eq:golden-generating-vector} is completely explicit and
available in arbitrary dimension.  Its special form also reduces the
number of distinct pair types: multiplication of both generating components
by the unit $r_m^{-(i-1)}$ modulo $11^m$ merely reindexes the points and
shows that the $(i,j)$-projection is identical, as an unlabeled rank-1
lattice point set, to the two-dimensional generator
\begin{align}\label{eq:golden-gap-projection}
 (1,r_m^{j-i})\pmod{11^m}.
\end{align}
Thus both its periodic and Euclidean projected geometry depend only on the
coordinate gap $j-i$.  This simplification is useful computationally,
although the family is not
claimed to minimize the worst projected mesh ratio.

\begin{remark}
\label{rem:nested-not-full-dimensional}
Any three elements $\beta_i,\beta_j,\beta_k$ lie in the two-dimensional
$\QQ$-vector space $L$.  Hence there is a nonzero integer vector
$(h_i,h_j,h_k)$ such that
\[
 h_i\beta_i+h_j\beta_j+h_k\beta_k=0.
\]
Applying every residue map gives
\[
 h_i z_{m,i}+h_j z_{m,j}+h_k z_{m,k}\equiv0\pmod{p^m}
 \quad(m\geq1).
\]
Thus, each three-dimensional projected dual lattice contains a fixed nonzero vector, independent of $m$.  Equivalently, all points of that projection lie on the proper subtorus defined by $h_i x_i+h_j x_j+h_k x_k\equiv0\pmod1$.  Its Euclidean covering radius is
therefore bounded below by a positive constant independent of $m$, and the three-dimensional projection cannot be quasi-uniform as $m\to\infty$. The quadratic-field construction is designed specifically for uniform control of all bivariate projections.
\end{remark}


\section{Extremal projective coefficient configurations}\label{sec:optimization}

\Cref{thm:kronecker-main,thm:nested-main} apply to every admissible configuration of distinct projective directions, but the resulting constants can vary substantially. We now formulate their quantitative selection as a pairwise bottleneck problem. The algebraic criteria below certify boundedness for all initial segments or levels; finite-range geometric criteria are used to distinguish candidates more sharply.

\subsection{Finite projective search spaces}

For $B\ge1$, let $\mathcal R_B$ be the bounded-height set in \Cref{def:projective-configuration} and let $\mathcal V_{d,B}=\{V\subset\mathcal R_B:|V|=d\}$. An elementary count gives
\begin{align}\label{eq:projective-height-count}
 |\mathcal R_B|=4\sum_{n=1}^B\varphi(n)=\frac{12}{\pi^2}B^2+O(B\log B),
\end{align}
so a pool of $d$ directions is available with $B=O(\sqrt d)$. In the quadratic construction, if $r_1=\iota_{\mathfrak p}(\omega)\pmod p$, admissibility removes precisely the residue direction $[-r_1:1]\in\PP^1(\FF_p)$; we denote the remaining candidates by $\mathcal R_{B,\mathfrak p}$.

Let $X$ be a finite candidate set and let $\ell(v,w)$ be a symmetric pair cost. For $V\subset X$, define $\mathcal L_\ell(V)=\max_{\{v,w\}\subset V}\ell(v,w)$ and minimize this quantity over $|V|=d$. For any threshold $\tau$, feasible configurations are exactly the $d$-cliques in the graph joining pairs with $\ell(v,w)\le\tau$. Thus the finite minimax problem can be solved exactly by a threshold-clique search when the candidate pool is moderate, and by deterministic greedy and exchange procedures otherwise.

A field-independent guide is the projective angle. For primitive representatives $v=[a:b]$ and $w=[c:e]$, set
\[
 \vartheta(v,w)=\arccos\frac{|ac+be|}{\sqrt{a^2+b^2}\sqrt{c^2+e^2}},
 \quad
 \chi(v,w)=\frac1{\sin\vartheta(v,w)}.
\]
If $\delta(V)=\min_{v\ne w}\vartheta(v,w)$, then
\begin{align}\label{eq:projective-angular-extremal-value}
 \sup_{|V|=d}\delta(V)=\frac{\pi}{d},
 \quad
 \inf_{|V|=d}\max_{v\ne w}\chi(v,w)=\frac1{\sin(\pi/d)}.
\end{align}
Indeed, the $d$ cyclic gaps on $\PP^1(\RR)$ sum to $\pi$, and rational directions are dense. For $d=4$, the cross configuration $([1:0],[0:1],[1:1],[1:-1])$ attains the optimum. This angular criterion does not determine the mesh ratio, but it explains why approximately equispaced directions are natural initial candidates.

\subsection{The nested-lattice objective}

Fix $(L,\mathfrak p)$ and a finite level set $\mathcal I$. For $v=[a:b]\in\mathcal R_{B,\mathfrak p}$, write $\beta(v)=a+b\omega$ and $z_m(v)=\iota_m(\beta(v))$. For $v\ne w$, define
\begin{align}\label{eq:nested-finite-periodic-pair-cost}
 \ell^{\rm per}_{L,\mathfrak p,\mathcal I}(v,w)
 =\max_{m\in\mathcal I}\rho_{\rm per}\left(P_{p^m,(z_m(v),z_m(w))}\right),
 \quad
 \mathcal M^{\rm per}_{L,\mathfrak p,\mathcal I}(V)
 =\max_{\{v,w\}\subset V}\ell^{\rm per}_{L,\mathfrak p,\mathcal I}(v,w).
\end{align}
Because $z_m(v)$ is invertible, the pair is equivalent to the slope $a_m\equiv z_m(w)z_m(v)^{-1}\pmod{p^m}$. Gauss reduction of the corresponding basis in \eqref{eq:dual-congruence-lattice}, followed by \Cref{prop:exact-periodic-mesh-ratio}, evaluates every edge cost exactly up to a square root and without generating the point set.

The periodic objective has the universal benchmark
\begin{align}\label{eq:planar-lattice-benchmark}
 \rho_{\rm per}\ge\frac{2}{\sqrt3},
\end{align}
with equality exactly for the triangular lattice shape. To see this, write $u=\|\bsb_2\|_2/\|\bsb_1\|_2\ge1$ and $t=(\bsb_1\cdot\bsb_2)/\|\bsb_1\|_2^2\in[0,1/2]$ for a Gauss-reduced basis. The formula in \Cref{prop:exact-periodic-mesh-ratio} becomes $u\sqrt{u^2+1-2t}/\sqrt{u^2-t^2}$, whose minimum is attained at $(u,t)=(1,1/2)$.

The algebraic certificate for a pair is $D_L(v,w)=\prod_{\ell=1}^2(|\sigma_\ell(\beta(v))|+|\sigma_\ell(\beta(w))|)$. By \Cref{prop:quadratic-dual-minimum} and \Cref{cor:dual-minimum-quasi-uniformity}, $\max_{\{v,w\}\subset V}D_L(v,w)$ yields an all-level upper bound for the projected mesh ratios. It is inexpensive and rigorous, but the exact periodic objective \eqref{eq:nested-finite-periodic-pair-cost} is usually much more discriminating.

\subsection{The Kronecker objective}

Fix $K=\QQ(\theta)$ and write $\alpha(v)=a\theta+b\theta^2$ for $v=[a:b]$. Direct periodic or Euclidean mesh ratios can be computed for a finite set $\mathcal N$ of initial-segment lengths, but doing so for every pair in a large pool is expensive. We therefore first use the truncated Diophantine quantities
\begin{align*}
 c_{{\rm sim},Q}(v,w) & = \min_{1\le q\le Q}q^{1/2}\max\{\|q\alpha(v)\|,\|q\alpha(w)\|\},\\
 c_{{\rm dual},H}(v,w) & = \min_{0<\|(h_1,h_2)\|_\infty\le H}\|(h_1,h_2)\|_\infty^2\|h_1\alpha(v)+h_2\alpha(w)\|,
\end{align*}
and the loss
\begin{align}\label{eq:kronecker-diophantine-pair-loss}
 \ell^{\rm dio}_{K;Q,H}(v,w)=\max\{c_{{\rm sim},Q}(v,w)^{-1},c_{{\rm dual},H}(v,w)^{-1}\}.
\end{align}
This score reflects the roles of simultaneous separation and dual covering. It is a screening criterion rather than an infinite-horizon certificate. The latter is supplied by the algebraic cost $\mathcal C_K(V)$ from \Cref{sec:kronecker}; direct mesh ratios over $N\in\mathcal N$ are then used to refine a shortlist.

In both constructions, the search therefore has the same structure: enumerate bounded-height candidates, compute pair costs, solve the finite bottleneck problem, enlarge the training range for the best configurations, and finally evaluate the Euclidean mesh ratios in $[0,1)^2$. The periodic mesh ratio, when evaluated, gives a certified Euclidean bound through \Cref{lem:boundary-transfer}, although boundary effects can change the ranking of the final candidates. Computational details and results are given in \Cref{sec:numerics}.

\section{Numerical experiments}\label{sec:numerics}

Finally, we examine the quantitative effect of the projective coefficient
configuration. The experiments have the following two purposes. First, we solve the finite bottleneck problems introduced in \Cref{sec:optimization} over a common bounded-height pool and compare the resulting configurations with the elementary explicit families of \Cref{sec:kronecker,sec:nested}.
Second, for representative dimensions, we compute the periodic and
boundary-sensitive mesh ratios directly. All computations use the
canonical primitive representatives from \Cref{def:projective-configuration}, and no shift is applied to the point sets.

\subsection{Experimental setting}

For the cubic-field Kronecker construction, we take $\theta$ to be the plastic constant, that is, the real root of $x^3-x-1$, and use the candidate set $\mathcal R_8$, which contains $88$ rational directions. The pair loss is \eqref{eq:kronecker-diophantine-pair-loss} with $Q=10^4$ and $H=40$. For the quadratic-field nested lattice construction, we use the golden-ratio field \eqref{eq:golden-field-data}, the prime $p=11$, and the root $r_1=4\pmod {11}$. Excluding the candidates whose reduction modulo $11$ equals the unique non-unit residue direction leaves $79$ candidates of height at most $8$. The exact periodic objective \eqref{eq:nested-finite-periodic-pair-cost} is evaluated over the levels $m=1,\ldots,7$, corresponding to as many as $11^7=19\,487\,171$ points. For each ambient dimension $d\in\{4,8,12,20,24,32\}$, the finite minimax problem is solved by the threshold-clique formulation of \Cref{sec:optimization}.

We compare the optimized Kronecker configuration with the one-parameter family
\[
 V_d^{\rm lin}=\{[1:0],[1:1],\ldots,[1:d-1]\},
\]
The nested lattice reference is the Hensel--Korobov family of \Cref{cor:golden-nested-family}; its projective coefficients are consecutive Fibonacci numbers.  Both references are completely explicit and available in every dimension, but were not chosen to optimize the worst bivariate projection.

The Kronecker loss is a finite Diophantine screening criterion, not a mesh ratio and not an infinite-horizon certificate.  The latter is provided by \Cref{thm:kronecker-main}.  By contrast, every nested-lattice edge cost used in the optimization is the exact periodic mesh ratio at the selected levels. Direct geometric validation is carried out for $d=4,8$.  For the Kronecker sequences, all coordinate pairs are evaluated at $N\in\{128,256,512,1024,2048,4096\}$.  For the nested construction, all periodic ratios are evaluated at $m=1,\ldots,7$.  The Euclidean ratios in the ordinary square are evaluated for all pairs at $m=3$; at $m=4$ they are evaluated for all pairs when $d=4$ and for the periodically worst pair when $d=8$.  This last restriction avoids repeating a Voronoi computation with $14\,641$ sites for every one of the $28$ coordinate pairs.

\subsection{Finite bottleneck optimization}

\Cref{tab:numerical-dimension-summary} gives the optimized finite bottleneck values and the corresponding reference values. The factor columns report the ratio of the reference value to the optimized value. The Kronecker optimization reduces the finite Diophantine loss by factors between $2.53$ and $7.77$. For the nested lattice construction, the Hensel--Korobov family is already optimal for the present finite problem when $d=4$, whereas the optimized configurations improve its worst periodic mesh ratio by factors between $1.69$ and $2.53$ for $d=8,\ldots,32$.

\begin{table}[H]
\centering
\caption{Finite bottleneck values over the height-eight candidate pools.
The Kronecker columns contain the loss
$\mathcal L_{\ell^{\rm dio}_{K;Q,H}}$, whereas the nested columns contain
the maximum exact periodic mesh ratio over $m=1,\ldots,7$.}
\label{tab:numerical-dimension-summary}
\small
\begin{tabular}{r@{\quad}|rrr@{\quad}rrr}
\hline
dimension & \multicolumn{3}{c}{Kronecker proxy loss} & \multicolumn{3}{c}{Nested periodic mesh ratio}\\
$d$ & optimized & reference & factor & optimized & reference & factor\\
\hline
4 & 15.72 & 110.73 & 7.04 & 2.56 & 2.56 & 1.00\\
8 & 31.07 & 241.52 & 7.77 & 6.12 & 14.15 & 2.31\\
12 & 51.25 & 241.52 & 4.71 & 7.53 & 14.15 & 1.88\\
20 & 91.73 & 416.93 & 4.55 & 21.84 & 55.16 & 2.53\\
24 & 119.15 & 416.93 & 3.50 & 25.12 & 55.16 & 2.20\\
32 & 165.10 & 416.93 & 2.53 & 32.68 & 55.16 & 1.69\\
\hline
\end{tabular}
\end{table}

For example, the optimized configurations at $d=8$ are
\begin{align*}
 V_{K,8}^{\rm opt}
 &=\{[0:1],[1:-1],[2:-1],[1:-3],
      [2:-3],[3:-2],[6:5],[6:-5]\},\\
 V_{L,8}^{\rm opt}
 &=\{[1:0],[0:1],[2:1],[1:3],
      [3:-2],[4:-3],[5:-3],[7:-4]\}.
\end{align*}
They are neither the equally spaced real projective configuration nor the minimizers of the coarse algebraic certificates alone.  The arithmetic images of the directions therefore matter in addition to their real projective angles.  The increase of the optimized values with $d$ is also consistent with \eqref{eq:projective-angular-extremal-value}: as more directions are placed in a fixed projective line, some coefficient pair must become increasingly ill-conditioned.

\subsection{Direct geometric validation}

The finite Kronecker proxy is intended only to screen coefficient configurations.  We therefore compare it with direct Voronoi-based mesh ratios.  The left panel of \Cref{fig:direct-validation} shows the largest ratio over all bivariate projections for $d=8$.  Over the six validation lengths, the optimized construction has maximum periodic and Euclidean ratios $10.28$ and $16.35$, respectively, compared with $36.30$ and $59.49$ for $V_8^{\rm lin}$.  The improvement is less uniform for $d=4$: the maxima are $8.38$ and $15.76$ for the optimized configuration, against $13.04$ and $17.52$ for the reference.  In particular, minimizing the truncated Diophantine loss does not imply pointwise dominance at every individual value of $N$.

For the nested lattice construction, the right panel of \Cref{fig:direct-validation} displays the exact worst periodic ratio at each level.  When $d=8$, its maximum over $m=1,\ldots,7$ decreases from $14.15$ for the Hensel--Korobov family to $6.12$ for the optimized configuration.  At $m=4$, the Euclidean ratio of the periodically worst pair is $12.16$ for the optimized design and $28.24$ for the reference. For $d=4$, both constructions attain the same finite optimum; their largest periodic ratio over the seven levels is $2.56$, only about $2.22$ times the triangular-lattice benchmark.

\begin{figure}[t]
\centering
\begin{minipage}[t]{0.49\textwidth}
  \centering
  \includegraphics[width=\linewidth]{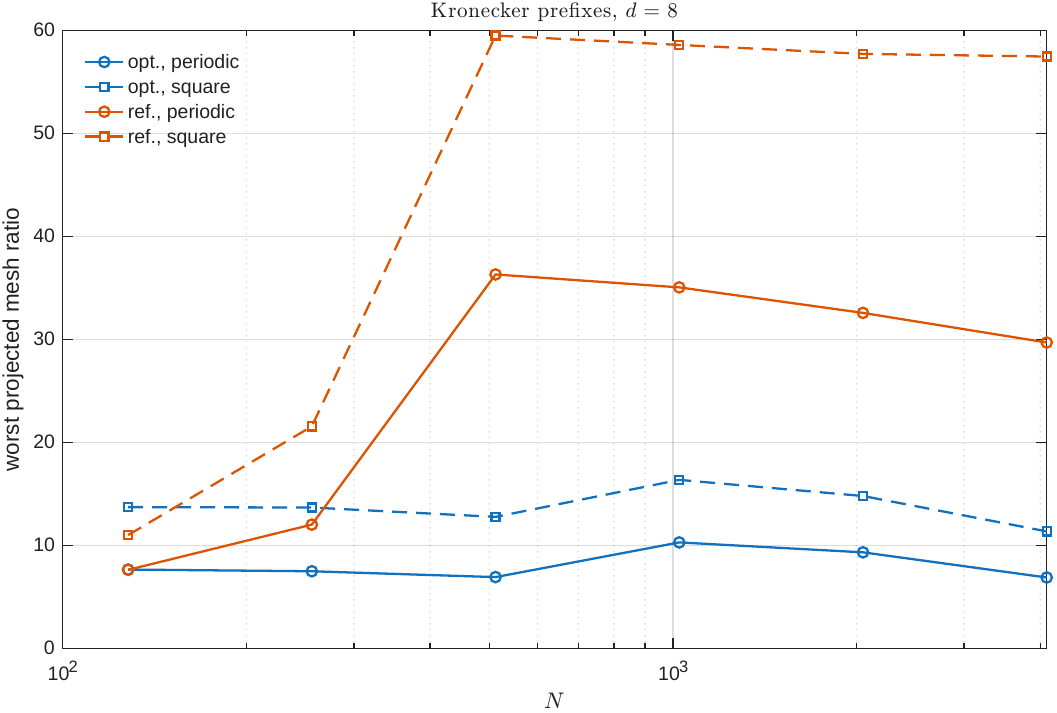}
  \smallskip
  \centerline{\small (a) Kronecker prefixes, $d=8$.}
\end{minipage}\hfill
\begin{minipage}[t]{0.49\textwidth}
  \centering
  \includegraphics[width=\linewidth]{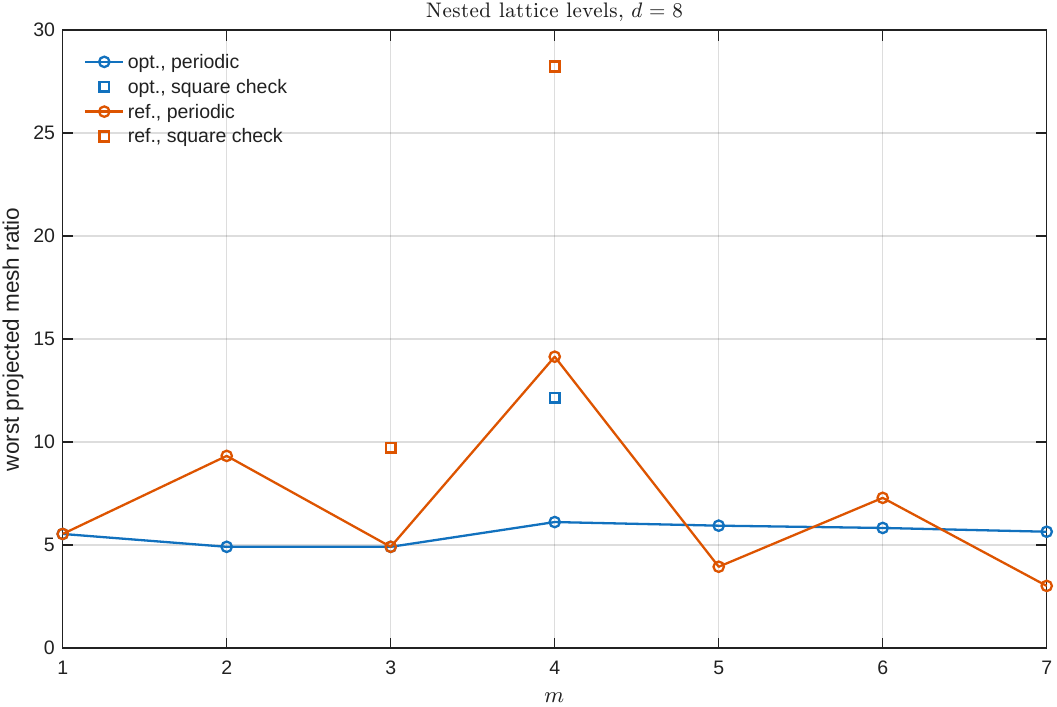}
  \smallskip
  \centerline{\small (b) Nested lattice levels, $d=8$.}
\end{minipage}
\caption{Direct geometric validation in dimension $d=8$. Every periodic and Euclidean value in the left panel is the maximum over all $28$ coordinate pairs. In the right panel, periodic values are also all-pair maxima; square markers at $m=3$ are all-pair maxima, whereas those at $m=4$ evaluate the pair that is worst for the periodic criterion.}
\label{fig:direct-validation}
\end{figure}

The Euclidean values are systematically larger than their periodic counterparts because the unit-square boundary removes neighboring lattice copies. Nevertheless, the optimized and reference constructions are ranked in the same order in the most pronounced $d=8$ comparisons.  This also illustrates the role of \Cref{lem:boundary-transfer}: periodic optimization controls the intrinsic lattice shape and supplies a uniform Euclidean bound, while a final boundary-sensitive check remains useful for quantitative selection.

\subsection{Bivariate projection geometry}

We finally visualize the bivariate geometry of the two constructions in dimension $d=8$. The purpose of these figures is complementary to the scalar worst-case summaries above: they show how the optimization changes the geometry across individual coordinate pairs.

\Cref{fig:kronecker-projection-matrix} displays all $\binom{8}{2}=28$ bivariate coordinate projections of the Kronecker constructions at $N=1024$. The upper triangular panels correspond to the one-parameter reference configuration $V_8^{\rm lin}$, whereas the lower triangular panels correspond to the optimized configuration $V_{K,8}^{\rm opt}$; the diagonal carries no projection. Panels symmetric about the diagonal therefore represent the same unordered coordinate pair, up to an interchange of the two coordinate axes. At this value of $N$, the worst periodic projection is the pair $(3,4)$ for the optimized configuration and the pair $(1,8)$ for the reference configuration.

It is important to emphasize that the optimization is minimax rather than pointwise. Thus, an individual projection of the reference configuration may be more isotropic than the corresponding optimized projection. The objective is instead to reduce the largest mesh ratio over all coordinate pairs. In this sense, the optimized configuration redistributes the unfavorable geometry among the projections and, in particular, avoids the most severe anisotropy responsible for the substantially larger worst-case values reported in the preceding subsection.

\begin{figure}[t]
 \centering
 \includegraphics[width=\textwidth]{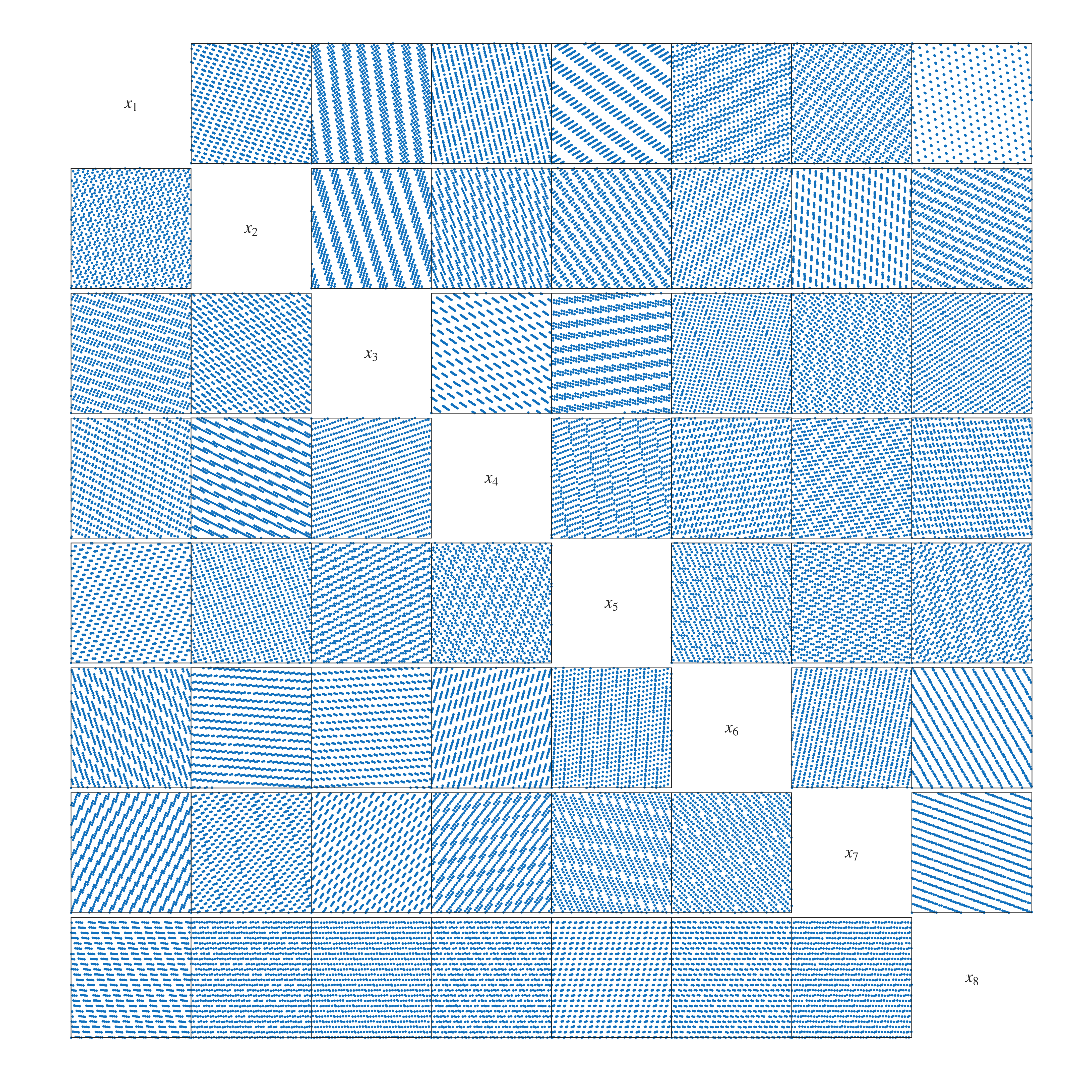}
 \caption{All bivariate coordinate projections of the Kronecker
 constructions in dimension $d=8$ at $N=1024$.  The upper triangular
 panels show the one-parameter reference configuration $V_8^{\rm lin}$,
 whereas the lower triangular panels show the optimized configuration
 $V_{K,8}^{\rm opt}$.  Corresponding panels across the diagonal represent
 the same coordinate pair with the two axes interchanged.}
 \label{fig:kronecker-projection-matrix}
\end{figure}

For the nested lattice construction, an analogous $8\times8$ matrix at level $m=4$ contains $11^4=14\,641$ points in every off-diagonal panel and is too dense to reveal the relevant lattice structure clearly. We therefore focus in \Cref{fig:nested-worst-projections} on the periodically worst projection of each configuration at this level.  The worst pair is $(2,8)$ for the optimized configuration and $(1,8)$ for the Hensel--Korobov reference. The reference projection is concentrated on a relatively small number of widely spaced parallel lines, whereas the optimized projection is distributed over a larger number of more closely spaced lines. This leads simultaneously to smaller uncovered gaps and better point separation, thereby reducing both the covering-to-separation imbalance and the resulting mesh ratio.

\begin{figure}[t]
\centering
\begin{subfigure}[t]{0.47\textwidth}
 \centering
 \includegraphics[width=\textwidth]{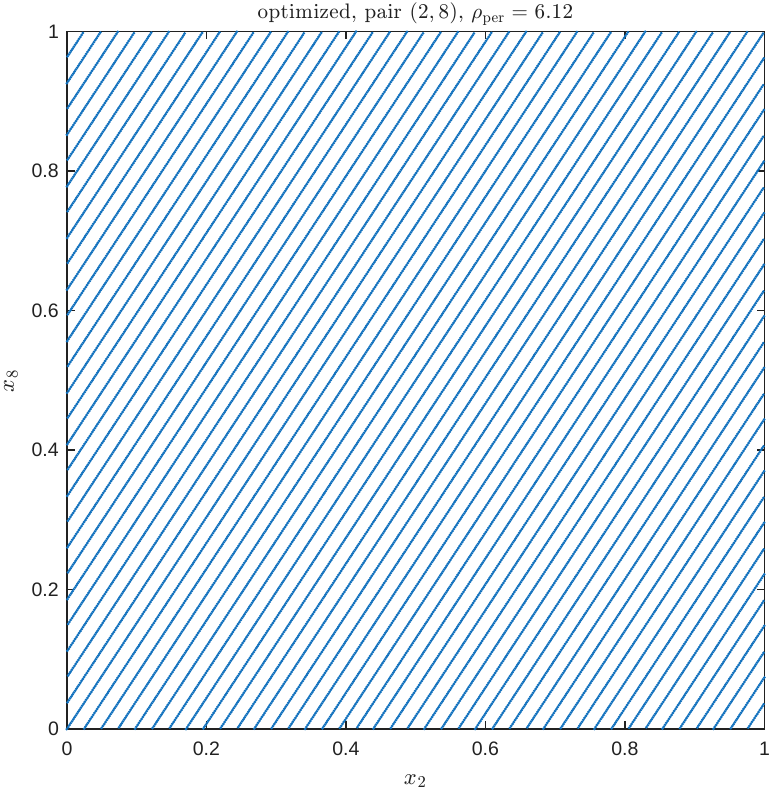}
 \caption{Optimized, pair $(2,8)$.}
\end{subfigure}\hfill
\begin{subfigure}[t]{0.47\textwidth}
 \centering
 \includegraphics[width=\textwidth]{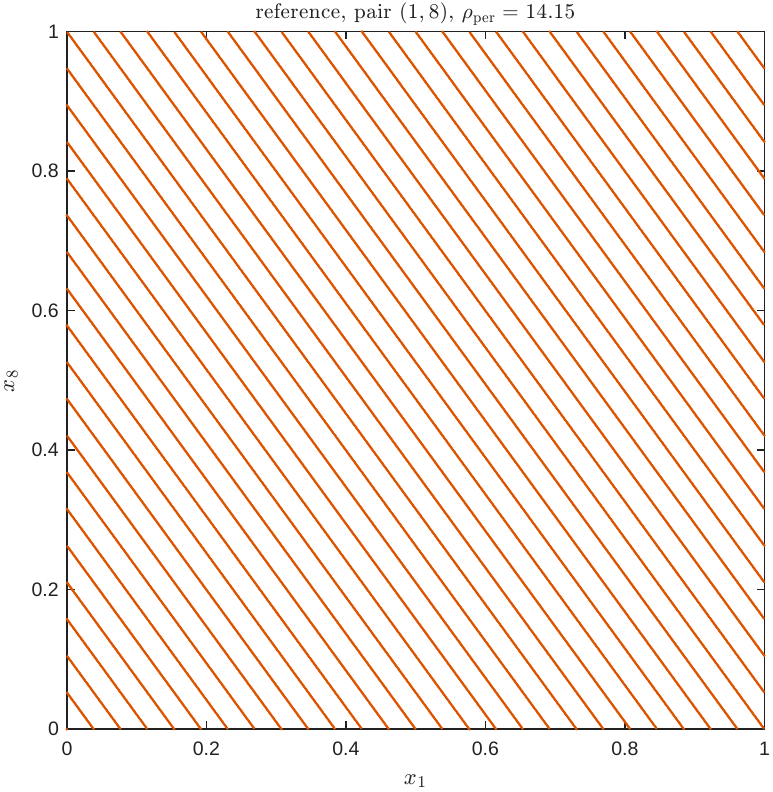}
 \caption{Hensel--Korobov reference, pair $(1,8)$.}
\end{subfigure}
\caption{Worst periodic bivariate projections of the nested lattice
constructions in dimension $d=8$ at level $m=4$, corresponding to
$11^4=14\,641$ points.  The periodic mesh ratios are $6.12$ for the
optimized configuration and $14.15$ for the reference configuration.}
\label{fig:nested-worst-projections}
\end{figure}

\FloatBarrier

\section{Concluding remarks}\label{sec:conclusion}

We have introduced two algebraic constructions of extensible point sets all of whose two-dimensional coordinate projections are quasi-uniform. The first construction uses a cubic number field to generate explicit Kronecker sequences, whereas the second uses a real quadratic field, a split prime, and a compatible $p$-adic embedding to generate nested rank-1 lattice point sets. In both cases, rational projective coefficient configurations provide a common parametrization of the coordinates. Besides simple fully explicit families, we formulated the quantitative selection of these coefficients as a finite bottleneck problem and constructed optimized configurations with substantially improved finite bottleneck values. The numerical experiments further showed, in representative dimensions, corresponding improvements in the worst-projection mesh ratios, both for the intrinsic periodic geometry and for the ordinary Euclidean geometry of the unit square.

The present paper has focused exclusively on two-dimensional coordinate projections. The underlying algebraic mechanism suggests a possible extension to higher-dimensional projections by using number fields of higher degree together with higher-dimensional rational coefficient configurations. For instance, one would naturally replace $\PP^1(\QQ)$ by an appropriate higher-dimensional projective space and impose non-degeneracy conditions on larger subsets of coefficient vectors. Establishing the corresponding Diophantine and lattice-geometric bounds, and determining effective constructions with favorable mesh-ratio constants, are beyond the scope of the present paper and are left for future work.

\section*{Declaration of generative AI use}

The author used ChatGPT (OpenAI, accessed in August 2026) for assistance with mathematical brainstorming, manuscript preparation, and the development and debugging of MATLAB code. All AI-assisted content was independently reviewed and verified by the author. The author assumes responsibility for all content.

\bibliographystyle{siam}
\bibliography{ref.bib}

@article {ABD06,
    AUTHOR = {Adamczewski, Boris and Bugeaud, Yann and Davison, Les},
     TITLE = {Continued fractions and transcendental numbers},
   JOURNAL = {Ann. Inst. Fourier (Grenoble)},
  FJOURNAL = {Universit\'e{} de Grenoble. Annales de l'Institut Fourier},
    VOLUME = {56},
      YEAR = {2006},
    NUMBER = {7},
     PAGES = {2093--2113},
       DOI = {10.5802/aif.2234},
}

@book {C57,
    AUTHOR = {Cassels, J. W. S.},
     TITLE = {An introduction to {D}iophantine Approximation},
    SERIES = {Cambridge Tracts in Mathematics and Mathematical Physics},
    VOLUME = {No. 45},
 PUBLISHER = {Cambridge University Press, New York},
      YEAR = {1957},
     PAGES = {x+166},
}

@article {DGLPS25,
    AUTHOR = {J. Dick and T. Goda and G. Larcher and F. Pillichshammer and K. Suzuki},
     TITLE = {On the quasi-uniformity properties of quasi-{M}onte {C}arlo point sets and sequences -- {P}art~{I}: {L}attices and {K}ronecker sequences},
   JOURNAL = {Math. Comp.},
  FJOURNAL = {Mathematics of Computation},
   YEAR = {to appear},
   NOTE = {Preprint available at \url{https://arxiv.org/abs/2502.06202}}
}

@article {DPW08,
    AUTHOR = {Dick, Josef and Pillichshammer, Friedrich and Waterhouse,
              Benjamin J.},
     TITLE = {The construction of good extensible rank-1 lattices},
   JOURNAL = {Math. Comp.},
  FJOURNAL = {Mathematics of Computation},
    VOLUME = {77},
      YEAR = {2008},
    NUMBER = {264},
     PAGES = {2345--2373},
       DOI = {10.1090/S0025-5718-08-02009-7},
}

@book {FLS06,
    AUTHOR = {Fang, K.-T. and Li, R. and Sudjianto, A.},
     TITLE = {Design and Modeling for Computer Experiments},
    SERIES = {Chapman \& Hall/CRC Computer Science and Data Analysis Series},
 PUBLISHER = {Chapman \& Hall/CRC, Boca Raton, FL},
      YEAR = {2006},
     PAGES = {xii+290},
}

@article {G24b,
    AUTHOR = {Goda, T.},
     TITLE = {One-dimensional quasi-uniform {K}ronecker sequences},
   JOURNAL = {Arch. Math.},
  FJOURNAL = {Archiv der Mathematik},
    VOLUME = {123},
      YEAR = {2024},
    NUMBER = {5},
     PAGES = {499--505},
       DOI = {10.1007/s00013-024-02039-0}
}

@article {H21,
    AUTHOR = {He, Xu},
     TITLE = {Lattice-based designs with quasi-optimal separation distance
              on all projections},
   JOURNAL = {Biometrika},
  FJOURNAL = {Biometrika},
    VOLUME = {108},
      YEAR = {2021},
    NUMBER = {2},
     PAGES = {443--454},
       DOI = {10.1093/biomet/asaa057},
}

@article {JK08,
    AUTHOR = {Joe, Stephen and Kuo, Frances Y.},
     TITLE = {Constructing {S}obol' sequences with better
              two-dimensional projections},
   JOURNAL = {SIAM J. Sci. Comput.},
  FJOURNAL = {SIAM Journal on Scientific Computing},
    VOLUME = {30},
      YEAR = {2008},
    NUMBER = {5},
     PAGES = {2635--2654},
       DOI = {10.1137/070709359},
}

@article {JMY90,
    AUTHOR = {Johnson, M. E. and Moore, L. M. and Ylvisaker, D.},
     TITLE = {Minimax and maximin distance designs},
   JOURNAL = {J. Statist. Plann. Inference},
  FJOURNAL = {Journal of Statistical Planning and Inference},
    VOLUME = {26},
      YEAR = {1990},
    NUMBER = {2},
     PAGES = {131--148},
       DOI = {10.1016/0378-3758(90)90122-B},
}

@article {JGB15,
    AUTHOR = {Joseph, V. Roshan and Gul, Evren and Ba, Shan},
     TITLE = {Maximum projection designs for computer experiments},
   JOURNAL = {Biometrika},
  FJOURNAL = {Biometrika},
    VOLUME = {102},
      YEAR = {2015},
    NUMBER = {2},
     PAGES = {371--380},
       DOI = {10.1093/biomet/asv002},
}

@article {MM95,
    AUTHOR = {Morris, Max. D. and Mitchell, Toby J.},
     TITLE = {Exploratory designs for computational experiments},
   JOURNAL = {J. Statist. Plann. Inference},
  FJOURNAL = {Journal of Statistical Planning and Inference},
    VOLUME = {43},
      YEAR = {1995},
    NUMBER = {3},
     PAGES = {381--402},
       DOI = {10.1016/0378-3758(94)00035-T},
}

@article {PM12,
    AUTHOR = {Pronzato, L. and M\"uller, W.~G.},
     TITLE = {Design of computer experiments: space filling and beyond},
   JOURNAL = {Stat. Comput.},
  FJOURNAL = {Statistics and Computing},
    VOLUME = {22},
      YEAR = {2012},
    NUMBER = {3},
     PAGES = {681--701},
       DOI = {10.1007/s11222-011-9242-3}
}

@article {PZ23,
    AUTHOR = {Pronzato, L. and Zhigljavsky, A.},
     TITLE = {Quasi-uniform designs with optimal and near-optimal uniformity
              constant},
   JOURNAL = {J. Approx. Theory},
  FJOURNAL = {Journal of Approximation Theory},
    VOLUME = {294},
      YEAR = {2023},
      NOTE = {Paper No. 105931, 14~pp.},
       DOI = {10.1016/j.jat.2023.105931}
}

@article {RHVD10,
    AUTHOR = {Rennen, Gijs and Husslage, Bart and Van Dam, Edwin R. and Den
              Hertog, Dick},
     TITLE = {Nested maximin {L}atin hypercube designs},
   JOURNAL = {Struct. Multidiscip. Optim.},
  FJOURNAL = {Structural and Multidisciplinary Optimization},
    VOLUME = {41},
      YEAR = {2010},
    NUMBER = {3},
     PAGES = {371--395},
       DOI = {10.1007/s00158-009-0432-y},
}

@article {SG26,
    AUTHOR = {N. Sakai and T. Goda},
     TITLE = {Space-filling lattice designs for computer experiments},
   JOURNAL = {arXiv preprint arXiv:2602.15390},
   YEAR = {2026},
}

@book {SWN03,
    AUTHOR = {Santner, T.~J. and Williams, B.~J. and Notz, W.~I.},
     TITLE = {The Design and Analysis of Computer Experiments},
    SERIES = {Springer Series in Statistics},
 PUBLISHER = {Springer-Verlag, New York},
      YEAR = {2003},
     PAGES = {xii+283},
}

@article {SW06,
    AUTHOR = {Schaback, R. and Wendland, H.},
     TITLE = {Kernel techniques: {F}rom machine learning to meshless methods},
   JOURNAL = {Acta Numer.},
  FJOURNAL = {Acta Numerica},
    VOLUME = {15},
      YEAR = {2006},
     PAGES = {543--639},
       DOI = {10.1017/S0962492906270016}
}

@book {W05,
    AUTHOR = {Wendland, H.},
     TITLE = {Scattered Data Approximation},
    SERIES = {Cambridge Monographs on Applied and Computational Mathematics},
    VOLUME = {17},
 PUBLISHER = {Cambridge University Press, Cambridge},
      YEAR = {2005},
}

\end{document}